\documentclass[english,12pt]{amsart}
\usepackage{amssymb}
\usepackage{amsfonts}
\usepackage{amscd}
\usepackage[all]{xy}
\usepackage{color}
\usepackage{enumitem}
\usepackage[utf8]{inputenc}
\usepackage{tikz}
\usepackage{todonotes}

\usepackage{pgf}
\usepackage[T1]{fontenc}
\usepackage{xr-hyper}
\usepackage[colorlinks=true, breaklinks=true, urlcolor= black, allcolors= black,bookmarksopen=true,linktocpage=true,plainpages=false,pdfpagelabels]{hyperref}
\usepackage{cleveref}

\CompileMatrices

\swapnumbers

\def\deg{{\rm deg}}

\def\11{{\mathbf 1}}
\def\AA{{\mathbb A}}

\def\FF{{\mathbb F}}

\def\PP{{\mathbb P}}

\def\RR{{\mathbb R}}

\def\ZZ{{\mathbb Z}}

\mathchardef\alphag="7C0B \mathchardef\betag="7C0C
\mathchardef\gammag="7C0D \mathchardef\deltag="7C0E
\mathchardef\varepsilong="7C22 \mathchardef\varphig="7C27
\mathchardef\psig="7C20 \mathchardef\zetag="7C10
\mathchardef\epsilong="7C0F \mathchardef\rhog="7C1A
\mathchardef\taug="7C1C \mathchardef\upsilong="7C1D
\mathchardef\iotag="7C13 \mathchardef\thetag="7C12
\mathchardef\pig="7C19 \mathchardef\sigmag="7C1B
\mathchardef\etag="7C11 \mathchardef\omegag="7C21
\mathchardef\kappag="7C14 \mathchardef\lambdag="7C15
\mathchardef\mug="7C16 \mathchardef\xig="7C18
\mathchardef\chig="7C1F \mathchardef\nug="7C17
\mathchardef\varthetag="7C23 \mathchardef\varpig="7C24
\mathchardef\varrhog="7C25 \mathchardef\varsigmag="7C26
\mathchardef\Omegag="7C0A \mathchardef\Thetag="7C02
\mathchardef\Sigmag="7C06 \mathchardef\Deltag="7C01
\mathchardef\Phig="7C08 \mathchardef\Gammag="7C00
\mathchardef\Psig="7C09 \mathchardef\Lambdag="7C03
\mathchardef\Xig="7C04 \mathchardef\Pig="7C05
\mathchardef\Upsilong="7C07

\newtheorem{thm}[subsection]{Theorem}
\newtheorem{lem}[subsection]{Lemma}
\newtheorem{cor}[subsection]{Corollary}

{Addendum}

\theoremstyle{definition}
\newtheorem{defn}[subsection]{Definition}

\newtheorem{def-prop}[subsection]{Proposition-Definition}
\newtheorem{def-theorem}[subsection]{Theorem-Definition}
\newtheorem{def-lem}[subsection]{Lemma-Definition}

\theoremstyle{remark}

\theoremstyle{plain}

\numberwithin{equation}{subsection}

\def\boxit#1#2{\setbox1=\hbox{\kern#1{#2}\kern#1}%
\dimen1=\ht1 \advance\dimen1 by #1 \dimen2=\dp1 \advance\dimen2 by
#1
\setbox1=\hbox{\vrule height\dimen1 depth\dimen2\box1\vrule}%
\setbox1=\vbox{\hrule\box1\hrule}%
\advance\dimen1 by .4pt \ht1=\dimen1 \advance\dimen2 by .4pt
\dp1=\dimen2 \box1\relax}

\renewcommand{\theequation}{\thesubsection.\arabic{equation}}

\mathchardef\alphag="7C0B \mathchardef\betag="7C0C
\mathchardef\gammag="7C0D \mathchardef\deltag="7C0E
\mathchardef\varepsilong="7C22 \mathchardef\varphig="7C27
\mathchardef\psig="7C20 \mathchardef\zetag="7C10
\mathchardef\epsilong="7C0F \mathchardef\rhog="7C1A
\mathchardef\taug="7C1C \mathchardef\upsilong="7C1D
\mathchardef\iotag="7C13 \mathchardef\thetag="7C12
\mathchardef\pig="7C19 \mathchardef\sigmag="7C1B
\mathchardef\etag="7C11 \mathchardef\omegag="7C21
\mathchardef\kappag="7C14 \mathchardef\lambdag="7C15
\mathchardef\mug="7C16 \mathchardef\xig="7C18
\mathchardef\chig="7C1F \mathchardef\nug="7C17
\mathchardef\varthetag="7C23 \mathchardef\varpig="7C24
\mathchardef\varrhog="7C25 \mathchardef\varsigmag="7C26
\mathchardef\Omegag="7C0A \mathchardef\Thetag="7C02
\mathchardef\Sigmag="7C06 \mathchardef\Deltag="7C01
\mathchardef\Phig="7C08 \mathchardef\Gammag="7C00
\mathchardef\Psig="7C09 \mathchardef\Lambdag="7C03
\mathchardef\Xig="7C04 \mathchardef\Pig="7C05
\mathchardef\Upsilong="7C07

\newcommand{\disc}{\operatorname{disc}}

\newcommand{\conv}{\operatorname{conv}}

\definecolor{immi}{rgb}{0,.6,.1}

\newbox\removebox
\newcommand\remove[1]{%
\setbox\removebox=\ifmmode\hbox{$#1$}\else\hbox{#1}\fi%
\leavevmode
\rlap{\textcolor{blue}{\vrule height0.8ex depth-0.6ex width\wd\removebox}}%
\box\removebox
}
\long\def\bigremove#1{%
\par\setbox\removebox=\vbox{#1}%
\vbox{%
\vbox to0pt{\hbox{\tikz\draw[color=blue,thick] (0,0) -- (\wd\removebox,-\ht\removebox)  (\wd\removebox,0) -- (0,-\ht\removebox);}}
\box\removebox
}
}

\definecolor{orange}{rgb}{1,0.5,0}

\newcommand{\private}[1]{\leavevmode{\scriptsize\color{blue}\marginpar{{\scriptsize Private comment}}#1\par}}
\renewcommand{\private}[1]{}

\title[Curves of fixed gonality with many rational points]{Curves of fixed gonality with many rational points}

\author[Vermeulen]{Floris Vermeulen}
\address{KU Leuven, Department of Mathematics, B-3001 Leu\-ven, Bel\-gium}
\email{floris.vermeulen@kuleuven.be}
\urladdr{https://sites.google.com/view/floris-vermeulen/homepage}

\begin{document}

\begin{abstract}
Given an integer $\gamma\geq 2$ and an odd prime power $q$ we show that for every large genus $g$ there exists a non-singular curve $C$ defined over $\FF_q$ of genus $g$ and gonality $\gamma$ and with exactly $\gamma(q+1)$ $\FF_q$-rational points. This is the maximal number of rational points possible. This answers a recent conjecture by Faber--Grantham. Our methods are based on curves on toric surfaces and Poonen's work on squarefree values of polynomials.
\end{abstract}

\maketitle

\section{Introduction}

In this article we study the maximal number of rational points on curves over finite fields. Curves will always be geometrically integral, but may be singular. Let $q$ be a prime power and fix a positive integer $g$. Define $N_q(g)$ to be the maximal number of $\FF_q$-rational points of a non-singular genus $g$ curve defined over $\FF_q$. The famous Weil bound yields that 
\[
N_q(g) \leq 2\sqrt{q}g + q + 1.
\]
By work of Ihara \cite{Ihara} and later Vladut and Drinfeld \cite{vladut_number_1983}, when $g$ is large compared to $q$, this bound can be improved to $N_q(g) \leq (\sqrt{q}-1+o(1))g$ as $g\to \infty$. To find lower bounds on the quantity $N_q(g)$ one has to construct non-singular curves with many rational points. Many approaches have been developed to deal with this problem, and we refer to the introduction of \cite{kresch_curves_2001} for an overview. By work of Elkies, Howe, Kresch, Poonen, Wetherell and Zieve \cite{elkies_curves_2004}, it is currently known that for every $q$ there exists some constant $c$ such that for every genus $g$
\[
cg < N_q(g).
\]
Now suppose that $C$ is a curve over $\FF_q$ equipped with a degree $\gamma$ map $C\to \PP^1$, also defined over $\FF_q$. Then since every $\FF_q$-rational point of $C$ must map to an $\FF_q$-rational point of $\PP^1$ we have
\begin{equation}\label{eq:gonality.bound}
\#C(\FF_q) \leq \gamma (q+1).
\end{equation}
In particular, this is true when $\gamma$ is equal to the gonality of $C$ over $\FF_q$. Recall that the \emph{gonality} of a curve $C$ is the minimal degree of a morphism defined over $\FF_q$ to $\PP^1$. For a positive integer $\gamma$, denote by $N_q(g, \gamma)$ the maximal number of $\FF_q$-rational points on a non-singular genus $g$ curve of gonality $\gamma$, defined over $\FF_q$. By convention, if no such curve exists we put $N_q(g, \gamma)=-\infty$. In \cite{van_der_geer_curves_2001}, van der Geer asks about the behaviour of this function $N_{q}(g, \gamma)$. Recently, this quantity has been studied for small $q$ and $g$ by Faber and Grantham in \cite{faber_binary_2020} and \cite{faber_ternary_2020}. They conjecture that for fixed $q, \gamma$ and $g$ sufficiently large, we have $N_q(g, \gamma) = \gamma(q+1)$. We prove this conjecture in odd characteristic.

\begin{thm}\label{thm:main.theorem}
Let $q$ be an odd prime power and fix a positive integer $\gamma\geq 2$. Then for all sufficiently large $g$, there exists a non-singular curve $C$ defined over $\FF_q$, of genus $g$ and gonality $\gamma$, and with $\gamma(q+1)$ $\FF_q$-rational points. In other words
\[
\lim_{g\to \infty} N_q(g, \gamma) = \gamma(q+1).
\]
\end{thm}

For the proof, we will use the theory of curves on toric surfaces. The idea is to construct the desired curve $C$ inside a certain toric surface $S$. Such a method was introduced by Kresch, Wetherell and Zieve in \cite{kresch_curves_2001} to construct non-singular curves over $\FF_q$ of every sufficiently large genus $g$ with at least $c_qg^{1/3}$ $\FF_q$-rational points, for some constant $c_q$ depending on $q$. In our setting, we also want to control the gonality of $C$. It would be desirable if $C$ would be smooth in the surface $S$. However, $S$ has only roughly $(q+1)^2$ $\FF_q$-rational points, so we cannot hope for this if $\gamma$ is large compared to $q$. Instead, we construct $C$ as a singular curve inside $S$. By carefully controlling the singularities of $C$, we are able to control both the rational points on $C$ and the (geometric) genus of $C$. To do this we rely on work by Poonen \cite{poonen_squarefree_2003} to make sure that a certain discriminant polynomial is squarefree. This is also where the condition that the characteristic is not 2 is needed.

In upcoming work together with Faber \cite{faber_vermeulen}, we use another approach using class field theory to try to prove Theorem \ref{thm:main.theorem}. In particular, we are able to prove that for fixed $q, \gamma$, we have that
\[
\limsup_{g\to \infty} N_q(g, \gamma) = \gamma(q+1)
\]
where one can assume that the resulting covers $C\to \PP^1$ are all abelian. In particular this holds when $q$ is even. Nevertheless, we also argue that abelian covers do not suffice to prove Theorem \ref{thm:main.theorem} in general.

\subsection*{Acknowledgements.} The author would like to thank Wouter Castryck and Xander Faber for helpful discussions and for comments on an earlier version of this article. The author thanks Art Waeterschoot for discussions around singular curves. The author was partially supported by KU Leuven IF C14/17/083, and partially by F.W.O. Flanders (Belgium) with grant number 11F1921N.

\section{Curves on toric surfaces}

We will construct the desired curve having many rational points as a singular curve in a toric surface. This section contains the relevant background material. We refer the reader to \cite{castryck_linear_2016} and \cite[Sec.\,3]{kresch_curves_2001} for more information.

Let $k$ be a perfect field and let $f\in k[t^{\pm 1},y^{\pm 1}]$ be a Laurent polynomial. Write $f=\sum_{i,j} c_{i,j} t^iy^j$. Then the \emph{Newton polygon} of $f$ is defined to be
\[
\Delta(f) = \conv\{(i,j)\mid c_{i,j} \neq 0\}\subseteq \RR^2,
\]
where $\conv$ denotes the convex hull. If $\Delta\subseteq \RR^2$ is any lattice polygon we denote by $\Delta^{(1)}$ the convex hull of its interior lattice points. Assume now that $f$ is absolutely irreducible and let $\tilde{C}$ be the non-singular projective model of the curve defined by $f=0$ in the torus $(k^\times)^2$. There is a strong relation between the combinatorics of $\Delta(f)$ and the geometry of $\tilde{C}$. The first result in this direction, proven by Baker, is that the genus of $\tilde{C}$ is bounded above by the number of interior lattice points of $\Delta(f)$, i.e.\
\[
g(\tilde{C}) \leq \# (\Delta(f)^{(1)} \cap \ZZ).
\]
We will call this \emph{Baker's bound}, see \cite{beelen}. The quantity $\# (\Delta(f)^{(1)} \cap \ZZ)$ has an interesting geometric interpretation. To explain, we introduce some toric geometry. Let $\Delta$ be a (2-dimensional) lattice polygon and consider the map
\[
\phi: (k^\times)^2 \to \PP^{\#(\Delta\cap \ZZ^2)-1}: (t,y)\mapsto (t^iy^j)_{(i,j)\in \Delta\cap \ZZ^2}.
\]
Define $S(\Delta)$ to be the closure of the image of $\phi$, this is a toric surface. Its fan is obtained by taking all primitive inward facing normals of $\Delta$. Denote the coordinates on $\PP^{\#(\Delta\cap \ZZ^2)-1}$ by $X_{i,j}$ for $(i,j)\in \Delta\cap \ZZ^2$. If $f = \sum_{i,j} c_{i,j}t^iy^j$ is an absolutely irreducible Laurent polynomial supported on $\Delta$ and each edge of $\Delta$ contains a lattice point from $\Delta(f)$ then the hyperplane section
\[
\sum_{(i,j)\in \Delta} c_{i,j} X_{i,j} = 0
\]
cuts out a curve $C$ in $S(\Delta)$. This curve is automatically birationally equivalent to $\tilde{C}$ by our assumptions. Then the quantity $\# (\Delta(f)^{(1)}\cap \ZZ^2)$ is the arithmetic genus $g_a(C)$ of the curve $C$, see \cite{khovanskii} and \cite[Lem.\,3.4]{kresch_curves_2001}. In particular, there is equality in Baker's bound if and only if $C$ is non-singular in $S(\Delta)$. 

\begin{defn}
Let $f(t,y)$ be an absolutely irreducible Laurent polynomial and let $\Delta\subseteq \RR^2$ be a lattice polygon. We say that $f$ is a \emph{$\Delta$-polynomial} if 
\begin{enumerate}
\item $\Delta(f)\subseteq \Delta$,
\item every edge of $\Delta$ contains a lattice point from $\Delta(f)$.
\end{enumerate} 
If moreover the genus of the curve defined by $f=0$ is equal to the number of interior lattice points of $\Delta$, then we call $f$ \emph{$\Delta$-toric}.
\end{defn}

By the remarks above, a $\Delta$-polynomial $f$ is $\Delta$-toric if and only if $C$ is non-singular in $S(\Delta)$.

Suppose that $f$ is a $\Delta$-polynomial for some lattice polygon $\Delta$ which is contained in the strip $\RR\times [0,\gamma]$ for some positive integer $\gamma$. Then there is naturally a morphism $C\to \PP^1$ of degree at most $\gamma$, obtained from the map
\[
V(f)\subseteq (k^\times)^2 \to k^\times: (y,t)\mapsto t.
\]
If $\Delta$ is not contained in any smaller horizontal strip of the form $\RR\times [a,b]$ then this map has degree equal to $\gamma$. 

\begin{lem}\label{lem:ratpoints.above.A1}
Let $\gamma$ be a positive integer and let $0 = k_0\leq k_1 < k_2 <  ...< k_\gamma$ be integers. Define $\ell_j = \sum_{i=0}^j k_i$ and let $f$ be an absolutely irreducible Laurent polynomial over a field $k$ with 
\[
\Delta(f) = \conv\{(0,0), (\ell_1, 1), (\ell_2, 2), \ldots, (\ell_\gamma, \gamma), (m_0, 0), (m_1, 1), \ldots, (m_\gamma, \gamma)  \},
\]
for certain integers $m_j > \ell_j$. Let $C\subseteq S(\Delta(f))$ be the curve defined by $f=0$ in $S(\Delta(f))$ and let $\pi: C\to \PP^1$ be the degree $\gamma$ map as above. Then $\pi^{-1}(0)$ consists of $\gamma$ distinct non-singular $k$-rational points. 
\end{lem}

Note that we are not saying anything about $C$ being non-singular above other points of $\PP^1$. See also \cite{kresch_curves_2001} for a similar argument to construct curves of a prescribed genus having many rational points. The shape of $\Delta(f)$ might look like figure \ref{fig:typical.Newton.polygon}. Note that the left side of this polygon consists of $\gamma$ distinct line segments, each containing two lattice points.

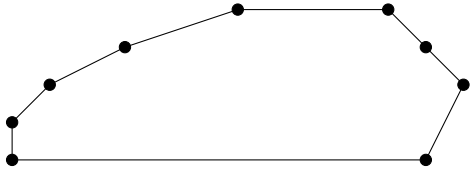
\begin{figure}[ht]
\centering
\begin{tikzpicture}[scale=0.5]
  \draw (0,0)--(0,1)--(1,2)--(3,3)--(6,4)--(10,4)--(11,3)--(12,2)--(11,0)--(0,0);
  
  \draw[fill] (0,0) circle [radius=0.15];
  \draw[fill] (0,1) circle [radius=0.15];
  \draw[fill] (1,2) circle [radius=0.15];
  \draw[fill] (3,3) circle [radius=0.15];
  \draw[fill] (6,4) circle [radius=0.15];
  \draw[fill] (10,4) circle [radius=0.15];
  \draw[fill] (11,3) circle [radius=0.15];
  \draw[fill] (12,2) circle [radius=0.15];
  \draw[fill] (11,0) circle [radius=0.15];
\end{tikzpicture}
\caption{A typical Newton polygon arising from Lemma \ref{lem:ratpoints.above.A1}}
\label{fig:typical.Newton.polygon}
\end{figure}

\begin{proof}
The edges of $\Delta = \Delta(f)$ correspond precisely to the one-dimensional torus invariant divisors of $S(\Delta)$. The points above $0$ of $\pi$ are precisely the intersections of $C$ with the torus invariant divisors corresponding to the edges on the left side of $\Delta$. We consider $S(\Delta)$ as sitting in $\PP^{\#(\Delta \cap \ZZ^2)-1}$ via the map $\phi$ as described above. Denote by $X_{i,j}$ the coordinates on $\PP^{\#(\Delta \cap \ZZ^2)-1}$, where $(i,j)\in \Delta\cap \ZZ^2$.

Fix such an edge $\tau$ on the left side of $\Delta$ and note that it contains exactly two lattice points, say $(\ell_j, j)$ and $(\ell_{j+1}, j+1)$. Let $D$ be the corresponding torus invariant divisor on $S(\Delta)$. Then $D$ is defined by taking $X_{i,j} = 0$ for $(i,j)\notin \tau$. Thus, on $D$, $C$ is defined by the hyperplane section
\[
aX_{\ell_j, j} + bX_{\ell_{j+1}, j+1}=0
\]
for certain $a,b\in k^{\times}$. Thus $C$ contains exactly one $k$-rational point on $D$. Moreover, this point is not invariant under the torus action, since both $a,b\neq 0$. Since there are $\gamma$ such edges on the left hand side of $\Delta$, we conclude that $\pi^{-1}(0)$ consists of $\gamma$ distinct $k$-rational points. The fact that these are all non-singular follows from the fact that $\pi$ has degree $\gamma$.
\end{proof}

\section{Poonen's theorem}

We recall here a theorem by Poonen on squarefree values of multivariate polynomials over $\FF_q[t]$. Fix a prime power $q$. For $a$ in $\FF_q[t]$ define $|a| = \#(\FF_q[t]/(a)) = q^{\deg a}$. For $S$ a subset of $\FF_q[t]^n$ and positive integers $d_1, ..., d_n$ define
\begin{align*}
S(d_1, ..., d_n) &= \{(f_1, ..., f_n) \in S \mid \deg f_i\leq d_i\}, \\
S(d_1, ..., d_n)' &= \{(f_1, ..., f_n) \in S \mid \deg f_i = d_i\}.
\end{align*}
We define the \emph{density} of $S$ to be 
\[
\mu(S) = \lim_{d_1, ..., d_n\to \infty} \frac{\#S(d_1, ..., d_n)}{\#\FF_q[t]^n(d_1, ..., d_n)} = \lim_{d_1, ..., d_n\to \infty} \frac{\#S(d_1, ..., d_n)}{q^{\sum_i (d_i+1)}},
\]
if the limit exists. Here the limit means that for every $\varepsilon>0$ there exists an $M$ such that if $d_1, ..., d_n\geq M$, then 
\[
\left|\frac{\#S(d_1, ..., d_n)}{\#\FF_q[t]^n(d_1, ..., d_n)}-\mu(S)\right|\leq \varepsilon. 
\]
Then Poonen's theorem states the following.

\begin{thm}[{{\cite{poonen_squarefree_2003}}}]\label{thm:poonen}
Let $F\in \FF_q[t][x_1, ..., x_n]$ be a polynomial which is squarefree when considered in $\FF_q(t)[x_1, ..., x_n]$. Let 
\[
S = \{a\in \FF_q[t]^n \mid F(a) \text{ is squarefree}\}.
\]
Then $\mu(S)$ exists and is equal to $\prod_p (1-c_p/|p|^{2n})$ where the product is over all non-zero primes $p$ of $\FF_q[t]$ and $c_p$ is the number of solutions of $F(x)=0$ in $\FF_q[t]/(p^2)$. Moreover, $\mu(S)>0$ if and only if $c_p < |p|^{2n}$ for all primes $p$.
\end{thm}

We will apply this theorem to the discriminant polynomial of a curve. However, we really want a statement about $S(d_1, ..., d_n)'$ rather than $S(d_1, ..., d_n)$.

\begin{cor}\label{cor:poonen}
Let $F\in \FF_q[t][x_1, ..., x_n]$ be a polynomial which is squarefree when considered in $\FF_q(t)[x_1, ..., x_n]$. Let 
\[
S = \{a\in \FF_q[t]^n \mid F(a) \text{ is squarefree}\}.
\]
Then
\[
\lim_{d_i\to \infty} \frac{\# S(d_1, ..., d_n)'}{q^{\sum_i d_i} (q-1)^n } = \mu(S).
\]
\end{cor}

\begin{proof}
The inclusion-exclusion principle gives that
\begin{align*}
\# S(d_1, ..., d_n)' &= \\
&  \# S(d_1, ..., d_n) - \#S(d_1-1, d_2, ..., d_n) - \#S(d_1, d_2-1, d_3, ..., d_n) - ...  \\
	& \quad + \#S(d_1-1, d_2-1, d_3, ..., d_n) + ... \pm \# S(d_1-1, ..., d_n-1).
\end{align*}
Dividing by $q^{\sum_i d_i}(q-1)^n$ and taking the limit $d_i\to \infty$ gives the desired statement.
\end{proof}

\section{Constructing curves with many points}

In this section we prove Theorem \ref{thm:main.theorem}. Let $\gamma\geq 2$ be the given gonality and let $q$ be an odd prime power. Let $g$ be a sufficiently large integer. We will construct a non-singular curve $C$ over $\FF_q$ of gonality $\gamma$ and of genus $g$ with $\gamma(q+1)$ rational points. The proof proceeds along several steps, which we first outline. 

\begin{enumerate}
\item Construct a family of polynomials in $\FF_q[t,y]$ among which we look for a suitable defining equation.
\item Use Poonen's theorem to a certain discriminant polynomial. This allows us to ensure that the curve is non-singular above $\AA^1(\overline{\FF_q}) \setminus \AA^1(\FF_q)$.
\item Construct a good candidate Newton polygon for the polynomial having a prescribed number of interior lattice points. This will give us control over the genus as well as the points above $\infty\in \PP^1$.
\item Construct the desired curve using the previous steps.
\item Check that the curve satisfies all required conditions: it is geometrically integral, non-singular, and has genus $g$, gonality $\gamma$ and $\gamma(q+1)$ $\FF_q$-rational points.
\end{enumerate}

We will look for a curve defined by an equation
\[
f(t,y,z) = \sum_{i=0}^\gamma f_i(t)y^iz^{\gamma-i},
\]
where $f_i(t)\in \FF_q[t]$. We will also write $f$ for the dehomogenized polynomial $f(t,y,1)$. Let $\Delta$ be the convex hull of $\Delta(f)$ together with $(0,0)$ and $(0,\gamma)$. We let $C$ be the curve defined by $f=0$ inside the surface $S(\Delta)=S$, which might not be geometrically integral. Recall that projection onto $t$ gives a degree $\gamma$ morphism $\pi: C\to \PP^1$, which extends to a morphism $S\to \PP^1$. Inside $S(\Delta)$, there is naturally a copy of $\AA^1\times \PP^1$ obtained as the union of the torus and all torus invariant divisors except those corresponding to edges on the right hand side of $\Delta$. If $C$ is geometrically integral, we denote by $\tilde{C}$ the non-singular model of $C$. The map $\pi$ naturally induces a degree $\gamma$ map $\tilde{C}\to \PP^1$, which we also denote by $\pi$.

\subsection*{Step 1.} We first ensure that $\tilde{C}$ has $\gamma$ distinct non-singular $\FF_q$-rational points above every point of $\AA^1(\FF_q)$. For this we rely on Lemma \ref{lem:ratpoints.above.A1}. So take integers $0 = k_0\leq k_1 < k_2 < ... < k_\gamma$, e.g.\ $k_i = i-1$ for $i>0$ will do, and put $\ell_j = \sum_{i=0}^j k_i$ and $L(\ell) = \sum_{j=1}^{\gamma-1} \ell_j$. Define $\alpha(t) = t^q-t = \prod_{a\in \FF_q}(t-a)\in \FF_q[t]$. We would like to take every $f_i(t)$ of the form
\[
\alpha(t)^{\ell_i} (1+\alpha(t)g_i(t)),
\]
for some polynomials $g_i(t)\in \FF_q[t]$. However, to make sure that $f$ is absolutely irreducible we need something more. Let $\beta(t)\in \FF_q[t]$ be an irreducible polynomial of degree $2$ (any degree $>1$ will do). Then we define
\[
f_i(t) =  \alpha(t)^{\ell_i} \beta(t)^{\delta_i} (1+\alpha(t)\beta(t)^{\delta_i'} g_i(t)),
\]
where $\delta_i = 1$ if $i<\gamma$ and $0$ if $i=\gamma$, and $\delta_i' = 1$ if $i=0, \gamma$ and $0$ if $i\neq 0, \gamma$. Note that then $f$ is a $\Delta$-polynomial, and also an Eisenstein polynomial with respect to $\beta$. We will want to pick the $g_i(t)$ such that the resulting curve $\tilde{C}$ has all of the desired properties. 

\subsection*{Step 2.} To ensure that the curve is non-singular at all points except above $\AA^1(\FF_q)$, we make sure that a certain discriminant polynomial is squarefree. First, we need a lemma.

\begin{lem}\label{lem:generic.discr.irre}
Let $k$ be a field of characteristic not $2$ and $\gamma$ a positive integer and let 
\[
D(x) = \disc_Y \left(\sum_{i=0}^\gamma x_iY^i \right) \in k[x_0, ..., x_\gamma]
\]
be the generic discriminant polynomial. Then $D(x)$ is irreducible in $k[x_0, ..., x_\gamma]$, so in particular it is squarefree.
\end{lem}
\begin{proof}
Consider the generic polynomial $p(x) = x^\gamma - s_1 x^{\gamma-1} + ... + (-1)^\gamma s_\gamma$, whose coefficients are the elementary symmetric polynomials $s_i$ in the roots $r_1, ..., r_\gamma$ of $p(x)$. Recall that the extension $k(r_1, ..., r_\gamma)/k(s_1, ..., s_\gamma)$ is Galois with Galois group $S_\gamma$. 

It is enough to prove that the discriminant polynomial
\[
D(s_1, ..., s_\gamma) = \prod_{i<j} (r_i - r_j)^2
\]
is irreducible in $k[s_1, ..., s_\gamma]$. So assume towards a contradiction that $D(s)$ factors as $D_1(s)D_2(s)$, for some non-constant polynomials $D_1, D_2\in k[s_1, ..., s_\gamma]$. By unique factorization there exist for every pair $i<j$ a $\varepsilon_{i,j}\in \{0,1,2\}$ such that $D_1(s) = \prod_{i<j} (r_i - r_j)^{\varepsilon_{i,j}}$. By Galois theory, one sees that $\varepsilon_{i,j}\leq \varepsilon_{i',j'}$ for every two pairs $i<j, i'<j'$. Hence we conclude that
\[
D_1(s) = c \prod_{i<j} (r_i-r_j),
\]
where $c\in k^\times$. But now, since $k$ has characteristic not $2$, switching $r_1$ and $r_2$ and keeping the other $r_i$ fixed changes the sign of $D_1$, implying that $D_1$ does not have coefficients in $k[s_1, ..., s_\gamma]$. This is a contradiction, and we conclude that $D(s)$ is irreducible.
\end{proof}

As a side remark, note that the above proof also shows that in characteristic 2, the generic discriminant polynomial is the square of an irreducible polynomial.

Define
\[
H(Y) = \sum_{i=0}^\gamma \alpha(t)^{\ell_i} \beta(t)^{\delta_i} (1+\alpha(t)\beta(t)^{\delta_i'} x_i) Y^i ,
\]
and
\[
F_1(x_0, \ldots, x_\gamma) = \disc_Y \left( \sum_{i=0}^\gamma \alpha(t)^{\ell_i} \beta(t)^{\delta_i} (1+\alpha(t)\beta(t)^{\delta_i'} x_i) Y^i\right),
\]
which is a polynomial over $\FF_q[t]$ in the variables $x_0, \ldots, x_\gamma$. Note that this is simply the discriminant of $f$ with respect to $y$, where we consider the $g_i$ as variables. We claim that $\alpha^{2L(\ell)}$ divides $F_1$. Fix a degree one prime $p$ of $\FF_q[t]$, i.e.\ a divisor of $\alpha$, and denote by $v_p$ the valuation corresponding to $p$ on $\FF_q(t)$. We work in some fixed algebraic closure $\overline{\FF_q(t, x_0, ..., x_\gamma)}$ of $\FF_q(t, x_0, ..., x_\gamma)$ and denote by $v_p$ also any extension of $v_p$ to this field, where we require that $v_p(x_i)\geq 0$ for any $i$. Now let $r_1, ..., r_\gamma$ be the roots of the polynomial $H(Y)$ in $\overline{\FF_q(t, x_0, ..., x_\gamma)}$. By looking at the Newton polygon of this polynomial with respect to the valuation $v_p$, we may assume, after reordering, that $v_p(r_i) = -k_i$. Hence we have that 
\[
v_p(F_1) = v_p\left( f_{\gamma}^{2\gamma-2} \prod_{i<j} (r_i-r_j)^2\right) = v_p(f_\gamma^{2\gamma-2}) - \sum_{i<j} k_j = 2L(\ell),
\]
and we conclude that $\alpha^{2L(\ell)}$ divides $F_1$. In fact, our argument shows that this is the exact power of $\alpha$ dividing $F_1$.

We use a similar argument with respect to $\beta$. Let $p$ be a linear factor of $\beta$ in $\overline{\FF_q}[t]$ and denote by $v_p$ an extension of the valuation to $\overline{\FF_q(t, x_0, ..., x_\gamma)}$, where $v_p(x_i)\geq 0$ for any $i$. Let $r$ be any root of $H(Y)$ in this field. Then $v_p(r) = 1/\gamma$ and $F_1 = \pm f_\gamma^{\gamma-2} N ((\partial_Y H)(r))$, where $N$ denotes the norm map from $\overline{\FF_q(t, x_0, ..., x_\gamma)}$ to $\FF_q(t, x_0, ..., x_\gamma)$. Now we have that
\[
\partial_Y H(r) = \gamma f_\gamma r^{\gamma-1} + (\gamma-1)f_{\gamma-1} r^{\gamma-2} + ... + f_1,
\]
and the terms here satisfy $v_p(if_ir^{i-1}) \geq (\gamma+i-1)/\gamma$ for $i= 1, ..., \gamma-1$. Therefore, we certainly have that $\beta^{\gamma-1}$ will divide $F_1$.

Now define
\[
F(x_0, \ldots, x_\gamma) = \frac{F_1(x_0, \ldots, x_\gamma)}{\alpha(t)^{2L(\ell)} \beta(t)^{\gamma-1}}\in \FF_q[t][x_0, \ldots, x_\gamma].
\]
We will apply Poonen's theorem to this polynomial. The highest degree part of $F_1$ is equal to
\[
 \disc_Y \left( \sum_{i=0}^\gamma \alpha(t)^{\ell_i+1} \beta(t)^{\delta_i+\delta_i'}  x_i Y^i\right).
\]
By Lemma \ref{lem:generic.discr.irre} the generic discriminant polynomial $\disc_Y(\sum_{i=0}^\gamma x_iY^i)$ is squarefree when considered in $\FF_q(t)[x_0, \ldots, x_{\gamma}]$, and so we conclude the same about $F_1$. Therefore, also $F$ is squarefree as an element of $\FF_q(t)[x_0, \ldots, x_\gamma]$. Define
\begin{align*}
S &= \{(g_i)_i\in \FF_q[t]^{\gamma+1}\mid F((g_i)_i) \text{ squarefree} \}.
\end{align*}
Then Poonen's theorem \ref{thm:poonen} states that the density $\mu(S)$ exists, and is equal to the product of $(1-c_p/|p|^{2\gamma+2})$ over all non-zero primes $p$ of $\FF_q[t]$, where $c_p$ is the number of zeroes of $F$ over $\FF_q[t]/(p^2)$. Moreover, if $c_p<|p|^{2\gamma+2}$ for all primes $p$ then $\mu(S)>0$.

If $p$ is a prime of degree $1$, fix any values of $g_i$ and let $r_1, ..., r_\gamma$ be the roots of $f=0$ in $\overline{\FF_q(t,y)}$. Let $v_p$ be an extension of the valuation corresponding to $p$ to this field with $v_p(y)\geq 0$. Then by a similar argument as above, we have
\[
v_p(F(g_0, ..., g_\gamma)) = v_p(F_1(g_0, ..., g_\gamma)) - 2L(\ell) = 0,
\]
so that $F$ cannot even have any zeroes modulo $p$. Hence $c_p = 0$. For $p=\beta$, take $g_0, ..., g_\gamma$ such that $v_\beta(f_1) = 1$. Then reasoning as above, one obtains that
\[
v_\beta(F(g_0, ..., g_\gamma)) = v_\beta(F_1(g_0, ..., g_\gamma)) - (\gamma-1) = \begin{cases} 0 & \text{ if } \gamma\neq 0 \text{ in } \FF_q \\
1 & \text{ if } \gamma = 0 \text{ in } \FF_q. \end{cases}
\]
Hence, we see that $F_1(g_0, ..., g_\gamma)$ is non-zero modulo $\beta^2$ and so $c_\beta < |\beta|^{2\gamma+2}$. If $p$ is another prime, not of degree $1$ or equal to $\beta$, then $\alpha(t)$ and $\beta(t)$ are both invertible modulo $p$. Then $F(x_0, \ldots, x_\gamma) \bmod p$ is obtained from the generic discriminant polynomial $\disc_Y(\sum_{i=0}^\gamma x_iY^i)\bmod p$ by linear substitution in the $x_i$. Since there exist squarefree polynomials of every degree over every finite field, we see that $c_p < |p|^{2\gamma+2}$. We conclude that $\mu(S)>0$, as desired.

\subsection*{Step 3.} We will want to take $(g_i)_i\in S(d_0, ..., d_{\gamma})'$ for suitable integers $d_i$ (depending on the genus $g$). Note that these $d_i$ also determine the Newton polygon of $f$. So we now look for a good candidate Newton polygon, which will give us control over the genus of the curve and moreover ensures that the curve is non-singular above $\infty$. 

Fix some residue class $g \equiv n \bmod \gamma-1$ of genera with $n=0, ..., \gamma-2$ and put $m = n + qL(\ell)$. Take $k_2'>\gamma-3$ such that
\[
k_2' \equiv \sum_{j=1}^{\gamma-3} j^2 - m\mod \gamma-1.
\]
Define $k_i' = \gamma-i$ for $i=3, ..., \gamma$ and put $k_1' = k_2'+1$. Note that
\[
k_1'>k_2'>...>k_\gamma', \quad \sum_j k_j' \geq 0, \quad \text{and } \sum_{j=1}^{\gamma-1} (\gamma-j)k_j' \equiv m\bmod \gamma-1.
\]
Define $\ell_j' = \sum_{i=1}^j k_i'$, $\ell_0'=0$ and for $r\geq 0$ a positive integer consider the lattice polygon
\[
\Delta_r = \conv\{(0,0), (0,\gamma), (r, \gamma), (r+\ell_1', \gamma-1), (r+\ell_2', \gamma-2), ..., (r+ \ell_\gamma', 0)\}.
\]
Our choice of integers $k_j'$ guarantees that the number of interior lattice points of $\Delta$ is congruent to $m$ modulo $\gamma-1$. Thus, for every sufficiently large integer $s$ there is some $r$ such that this polygon $\Delta_r$ has exactly $n+qL(\ell) + s(\gamma-1)$ interior lattice points. Note also that by our choice of $k_i'$, the right hand side of $\Delta$ consists of $\gamma$ line segments, each containing exactly two lattice points.

\subsection*{Step 4.} We now construct the desired curve. By Corollary \ref{cor:poonen} there exists some $M$ such that if $d_0, \ldots, d_\gamma \geq M$ then 
\[
S(d_0, \ldots, d_\gamma)' \neq \emptyset.
\]
Let $g$ be the desired genus for our curve and assume that $g\equiv n\bmod \gamma-1$. If $g$ is sufficiently large, then there exists some positive integer $r$ such that the polygon $\Delta = \Delta_r$ constructed in the previous step has exactly $g + qL(\ell)$ interior lattice points. Recall that $\deg \beta =2$ and define for every $i\in \{0, \ldots, \gamma\}$
\[
d_i = r + \ell_{\gamma-i}' - q(\ell_i+1) - 2(\delta_i+\delta_i').
\]
This will be the desired degree of $g_i(t)$. If $g$ is sufficiently large, then so is $r$ and hence $d_i\geq M$ for all $i$. Therefore, there exists some
\[
(g_0, \ldots, g_\gamma) \in S(d_0, \ldots, d_\gamma)'.
\]
Let $f(t,y)$ be the polynomial as constructed above from the $(g_i)_i$, namely we put
\[
f_i(t) =  \alpha(t)^{\ell_i} \beta(t)^{\delta_i} (1+\alpha(t)\beta(t)^{\delta_i'} g_i(t)),
\]
and define $f = \sum_{i=0}^\gamma f_i(t)y^i$. Let $C = V(f)$ inside $S = S(\Delta)$. We claim that this curve is as desired. Namely, $f$ is absolutely irreducible and the non-singular model of $C$ is of genus $g$ with $\gamma(q+1)$ $\FF_q$-rational points and gonality $\gamma$.

\subsection*{Step 5.} We first prove that $f$ is absolutely irreducible and that $C$ is smooth above $\AA^1(\overline{\FF_q})\setminus \AA^1(\FF_q)$. We need a general lemma.

\begin{lem}\label{lem:nonsingular.above.A1}
Let $k$ be a field and let $f(t, y,z) = \sum_{i=0}^\gamma f_i(t)y^iz^{\gamma-i}$ be a polynomial over $k$. Let $C = V(f)$ be the curve determined by $f=0$ inside $\AA^1\times \PP^1$, where the coordinate on $\AA^1$ is $t$ and the coordinates on $\PP^1$ are $y,z$. If $C$ has a singularity above a non-zero prime $p$ of $k[t]$ then $p^2$ divides $\disc_y f(t,y,1)$.
\end{lem}
\begin{proof}
By invariance of the discriminant under translations and taking reciprocals we can assume that $C$ has a singularity at $t=a, y=0, z=1$, where $p(a)=0$. Here we have to move to some algebraic closure of $k$ but this is not a problem. Now, being a singularity means that
\[
f(a, 0) = f_0(a) = 0, (\partial_t f)(a, 0) = f_0'(a)=0, (\partial_y f)(a, 0) = f_1(a) = 0.
\]
By computing the discriminant from the Sylvester matrix of $f, \partial_y f$, one sees that $(t-a)^2$ will divide $\disc_y f(t,y)$. Since the original $f$ is defined over $k$, this situation occurs for any root $a$ of $p$ and so $p^2$ will divide $\disc_y f(t,y)$.
\end{proof}

The condition of having squarefree discriminant at a prime $p$ is stronger than being non-singular. Geometrically, having squarefree discriminant at $p$ means that the ramification of the map $C\to \AA^1$ above $p$ is of the form $(2, 1, ..., 1)$ or $(1,...,1)$. 

\begin{lem}\label{lem:abs.irre}
The polynomial $f$ is absolutely irreducible and $C$ has no singularities above $\PP^1(\overline{\FF_q})\setminus \AA^1(\FF_q)$.
\end{lem}
\begin{proof}
Suppose that $f(y,t) = a(y,t)b(y,t)$ for certain $a, b\in \overline{\FF_q}[t,y]$. Write $a = \sum_{i=0}^\lambda a_i(t) y^i, b = \sum_{i=0}^{\gamma-\lambda} b_i(t)y^i$ where $\lambda$ is the $y$-degree of $a$, and $a_i(t), b_i(t)\in \overline{\FF_q}[t]$. Let $p$ be a linear factor of $\beta$ in $\overline{\FF_q}[t]$ and denote by $v_p$ the induced valuation on $\overline{\FF_q}[t]$. Then we have that 
\[
1 = v_p(f_0) = v_p(a_0) + v_p(b_0), \quad 0 = v_p(f_\gamma) = v_p(a_\lambda) + v_p(b_{\gamma-\lambda}).
\]
So without loss of generality, $v_p(a_\lambda) = v_p(b_{\gamma-\lambda}) = v_p(b_0) = 0$ and $v_p(a_0) = 1$. Now take $\lambda'\geq 0$ such that $v_p(a_{\lambda'})=0$ but $v_p(a_i) \geq 1$ for $i=0, ..., \lambda'-1$. Then
\[
v_p(f_{\lambda'}) = v_p(b_{0}a_{\lambda'} + b_{1}a_{\lambda'-1} + \ldots) = 0
\]
and by construction of $f$, $\lambda' = \gamma$ and hence also $\lambda = \gamma$. This implies that $b = b_0(t)\in \overline{\FF_q}[t]$. It is clear that $b_0$ is coprime to both $\alpha(t)$ and $\beta(t)$, because $\alpha(t)$ and $\beta(t)$ do not divide $f(y,t)$. But then $b_0(t)^{2\gamma-2}$ would divide $\disc_y f(t,y)$, which contradicts the fact that $(g_i)_i\in S$. So we conclude that $f$ is absolutely irreducible.

That $C$ has no singularities above primes $p$ not dividing $\alpha(t)$ or $\beta(t)$ follows directly from Lemma \ref{lem:nonsingular.above.A1} and the fact that $(g_i)_i\in S$. That the same holds above $\beta$ follows from the fact that $f$ is an Eisenstein polynomial with respect to $\beta$. Finally, by an argument similar to Lemma \ref{lem:ratpoints.above.A1} but using the right hand side of $\Delta$, we see that the fibre of $C$ above $\infty$ consists of $\gamma$ non-singular $\FF_q$-rational points.
\end{proof}

Denote by $\tilde{C}$ the non-singular model of $C$. As for the genus of $\tilde{C}$, and the points above $\PP^1(\FF_q)$ we have the following.

\begin{lem}\label{lem:control.genus}
Every fibre of $\pi: \tilde{C}\to \PP^1$ above $\PP^1(\FF_q)$ consists of $\gamma$ distinct non-singular $\FF_q$-rational points. Moreover, the genus of $\tilde{C}$ is equal to
\[
g(\tilde{C}) = \# (\Delta^{(1)}\cap \ZZ^2) - qL(\ell) = g.
\]
\end{lem}

\begin{proof}
The first statement for fibres above $\AA^1(\FF_q)$ follows directly from Lemma \ref{lem:ratpoints.above.A1} and our construction of $f$. Note that the right hand side $\Delta$ consists of $\gamma$ line segments, each containing two lattice points. Above $\infty$, a similar argument as Lemma \ref{lem:ratpoints.above.A1} using the right hand side of $\Delta$ then gives that also $\pi^{-1}(\infty)$ consists of $\gamma$ non-singular $\FF_q$-rational points.

By the previous lemma, $C$ has no singularities in $S$, except above points of $\AA^1(\FF_q)$. The curve $C$ has arithmetic genus $\# (\Delta^{(1)}\cap \ZZ^2)$. Recall that the torus-invariant points of $S$ correspond to the lattice points of $\Delta$, so let $p\in S$ be the torus-invariant point corresponding to $(0,\gamma)\in \Delta$. By looking at the fans of $S$ and $S' = S(\Delta(f))$, one sees that there is a toric morphism $\psi: S'\to S$ which is an isomorphism away from $p$. The strict transform of $C$ under this morphism is the curve $C'$ defined by $f$, but in the surface $S'$, and $\psi$ restricts to a morphism $C'\to C$ which is an isomorphism away from $p$. Since $\pi(p) = 0$, the curve $C$ is isomorphic to $C'$ above $\PP^1\setminus \{0\}$. By Lemma \ref{lem:ratpoints.above.A1}, the curve $C'$ is non-singular above $0$, and by comparing Newton polygons, the arithmetic genus of $C'$ is equal to
\[
\# (\Delta(f)^{(1)}\cap \ZZ^2) = g_a(C) - L(\ell).
\]
The same reasoning holds above every other point of $\AA^1(\FF_q)$, by our construction of the polynomial $f$. After resolving the singularities above every $a\in \AA^1(\FF_q)$, we end up with the curve $\tilde{C}$, since these were the only singularities of $C$. Hence we conclude that this curve has genus
\[
\# (\Delta(f)^{(1)}\cap \ZZ^2) - q L(\ell)=g. \qedhere
\]
\end{proof}

This lemma implies immediately that $\tilde{C}$ has $\gamma(q+1)$ $\FF_q$-rational points. Finally, the fact that $\tilde{C}$ has gonality $\gamma$ follows directly from the gonality bound \ref{eq:gonality.bound} since $\tilde{C}$ has $\gamma(q+1)$ $\FF_q$-rational points and we have an $\FF_q$-rational map $\tilde{C}\to \PP^1$ of degree $\gamma$.

\bibliographystyle{amsplain}
\bibliography{MyLibrary}

\end{document}